\documentclass{amsart}
\usepackage{amssymb, amsmath, mathrsfs, verbatim, url, bbm}

\theoremstyle{plain}
\newtheorem{theorem}{Theorem}
\newtheorem{lemma}[theorem]{Lemma}
\newtheorem{proposition}[theorem]{Proposition}
\newtheorem{corollary}[theorem]{Corollary}

\theoremstyle{definition}

\newcommand{\re}{\upharpoonright}
\newcommand{\Cof}{\mathsf{Cof}}
\newcommand{\Fin}{\mathsf{Fin}}
\newcommand{\DD}{\mathcal{D}}
\newcommand{\FF}{\mathcal{F}}
\newcommand{\GG}{\mathcal{G}}

\newcommand{\TT}{\mathcal{T}}

\newcommand{\XX}{\mathcal{X}}
\newcommand{\YY}{\mathcal{Y}}
\newcommand{\QQQ}{\mathbb{Q}}

\begin{document}

\title{Every filter is homeomorphic to its square}

\author{Andrea Medini}
\address{Kurt G\"odel Research Center for Mathematical Logic
\newline\indent University of Vienna
\newline\indent W\"ahringer Stra{\ss}e 25
\newline\indent A-1090 Wien, Austria}
\email{andrea.medini@univie.ac.at}
\urladdr{http://www.logic.univie.ac.at/\~{}medinia2/}

\author{Lyubomyr Zdomskyy}
\address{Kurt G\"odel Research Center for Mathematical Logic
\newline\indent University of Vienna
\newline\indent W\"ahringer Stra{\ss}e 25
\newline\indent A-1090 Wien, Austria}
\email{lyubomyr.zdomskyy@univie.ac.at}
\urladdr{http://www.logic.univie.ac.at/\~{}lzdomsky/}


\thanks{The first-listed author was supported by the FWF grant M 1851-N35. The second-listed author was supported by the FWF grant I 1209-N25.}

\date{May 13, 2016}

\begin{abstract}
We show that every filter $\FF$ on $\omega$, viewed as a subspace of $2^\omega$, is homeomorphic to $\FF^2$. This generalizes a theorem of van Engelen, who proved that this holds for Borel filters.
\end{abstract}

\maketitle

\section{Introduction}

In \cite{vanengeleni}, van Engelen obtained a purely topological characterization of filters, among the zero-dimensional Borel spaces.\footnote{Actually, van Engelen stated his results for ideals. Using the homeomorphism $c:2^\omega\longrightarrow 2^\omega$ defined by $c(X)(n)=1-X(n)$ for $X\in 2^\omega$ and $n\in\omega$, one sees that his results also hold for filters.} In particular, he obtained the following result (see \cite[Lemma 3.1]{vanengeleni}).
\begin{theorem}[van Engelen]\label{borelfilters}
If $\FF$ is a Borel filter then $\FF$ is homeomorphic to $\FF^2$.
\end{theorem}
The main ingredients of his proof are the fact that every filter $\FF$ is Wadge equivalent to $\FF^2$ (which is easy to see using the operation of intersection), a theorem of Steel from \cite{steel}, and some of his previous work from \cite{vanengelent}. It is natural to ask whether the assumption that $\FF$ is Borel is really necessary in Theorem \ref{borelfilters}. Our main result shows that it is not (see Theorem \ref{main}), and its proof only uses elementary methods.

\section{Notation}

Throughout this paper, $\Omega$ will denote a countably infinite set.
A \emph{filter} on $\Omega$ is a collection $\FF$ of subsets of $\Omega$ that satisfies the following conditions. We will write $X\subseteq^\ast Y$ to mean that $X\setminus Y$ is finite, and we will write $X=^\ast Y$ to mean that $X\subseteq^\ast Y$ and $Y\subseteq^\ast X$.
\begin{enumerate}
\item\label{empty} $\varnothing\notin\FF$ and $\Omega\in\FF$.
\item\label{finitemod} If $X\in\FF$ and $X=^\ast Y\subseteq\Omega$ then $Y\in\FF$.
\item\label{superset} If $X\in\FF$ and $X\subseteq Y\subseteq\Omega$ then $Y\in\FF$.
\item\label{intersection} If $X,Y\in\FF$ then $X\cap Y\in\FF$.
\end{enumerate}
All filters are assumed to be on $\omega$ unless we explicitly say otherwise. A filter is \emph{principal} if there exists $\Omega\subseteq\omega$ such that $\FF=\{X\subseteq\omega: \Omega\subseteq^\ast X\}$. Define $\Fin(\Omega)=\{X\subseteq\Omega:X\text{ is finite}\}$ and $\Cof(\Omega)=\{X\subseteq\Omega:\Omega\setminus X\text{ is finite}\}$.

We will freely identify any collection $\XX$ consisting of subsets of $\Omega$ with the subspace of $2^\Omega$ consisting of the characteristic functions of elements of $\XX$. In particular, every filter on $\Omega$ will inherit the subspace topology from $2^\Omega$.

Given a function $f$ and a subset $S$ of the domain of $f$, let $f[S]=\{f(X):X\in S\}$ denote the image of $S$ under $f$.

By \emph{space} we will always mean separable metrizable topological space. A space is \emph{crowded} if it is non-empty and it has no isolated points. Given spaces $\XX$ and $\YY$, we will write $\XX\approx \YY$ to mean that $\XX$ and $\YY$ are homeomorphic. We will be using freely the following well-known characterizations of $\QQQ$ and $2^\omega$ (see \cite[Theorem 1.9.6]{vanmill} and \cite[Theorem 1.5.5]{vanmill} respectively).
\begin{itemize}
\item If $\XX$ is a crowded countable space then $\XX\approx\QQQ$.
\item If $\XX$ is a crowded compact zero-dimensional space then $\XX\approx 2^\omega$.
\end{itemize}
We will also assume that the reader is familiar with the basic theory of topologically complete spaces (see for example \cite[Section A.6]{vanmill}).

Given a collection $\XX$ consisting of subsets of $\omega$ and $\Omega\subseteq\omega$, define
$$
\XX\re\Omega=\{X\cap\Omega:X\in\XX\}.
$$
Notice that $\FF\re\Omega=\{X\in\FF:X\subseteq\Omega\}$ whenever $\FF$ is a filter and $\Omega\in\FF$.

We conclude this section by remarking that many authors (including van Engelen in \cite{vanengeleni}) give a more general notion of filter than the one we gave above. The most general notion possible seems to be the following. Define a \emph{prefilter} on $\Omega$ to be a collection $\FF$ of subsets of $\Omega$ that satisfies conditions $(\ref{superset})$ and $(\ref{intersection})$. The next proposition, which can be safely assumed to be folklore, shows that our definition does not result in any substantial loss of generality.

\begin{proposition}
Let $\GG$ be an infinite prefilter on $\omega$. Then either $\GG\approx 2^\omega$ or $\GG\approx\FF$ for some filter $\FF$.
\end{proposition}
\begin{proof}
Let $\Omega=\omega\setminus\bigcap\GG$, and observe that $\Omega$ is infinite because $\GG$ is infinite.
Notice that $\GG\re\Omega$ is a prefilter on $\Omega$. First assume that $\varnothing\in\GG\re\Omega$.
This means that $\bigcap\GG=\omega\setminus\Omega\in\GG$, hence $\GG=\{X\subseteq\omega:\bigcap\GG\subseteq X\}\approx 2^\omega$.

Now assume that $\varnothing\notin\GG\re\Omega$. We claim that $\GG\re\Omega$ is in fact a filter on $\Omega$.
In order to prove this claim, it only remains to show that condition $(\ref{finitemod})$ is satisfied. Notice that it will be enough to show that $\Cof(\Omega)\subseteq\GG\re\Omega$. So let $F\in\Fin(\Omega)$. Since $\Omega=\omega\setminus\bigcap\GG$ and $\GG$ satisfies condition $(\ref{intersection})$, there must be $X\in\GG$ such that $X\subseteq\omega\setminus F$.
It follows that $\omega\setminus F\in\GG$, hence $\Omega\setminus F\in\GG\re\Omega$. Finally, it is straightforward to check that $\GG\approx\GG\re\Omega$.
\end{proof}

\section{Preliminary results}

The following three lemmas will be needed in the proof of Theorem \ref{main}.

\begin{lemma}\label{restriction}
Assume that $\FF$ is a non-principal filter and $\Omega\in\FF$.
Then $\FF\re\Omega\approx\FF$. 
\end{lemma}
\begin{proof}
Fix $\Omega^\ast\subseteq\Omega$ such that $\Omega^\ast\in\FF$ and $\Omega\setminus\Omega^\ast$ is infinite. This is possible because $\FF$ is non-principal. Fix a bijection $\sigma:\omega\setminus\Omega^\ast\longrightarrow\Omega\setminus\Omega^\ast$ and let $\tau:\Omega^\ast\longrightarrow\Omega^\ast$ be the identity. Set $\pi=\sigma\cup\tau$ and notice that $\pi:\omega\longrightarrow\Omega$ is a bijection. Therefore, the function $h:2^\omega\longrightarrow 2^\Omega$ defined by setting $h(X)=\pi[X]$ is a homeomorphism. Furthermore, using the fact that $\Omega^\ast\in\FF$, it is straightforward to check that $h[\FF]=\FF\re\Omega$. This shows that $\FF\approx\FF\re\Omega$.
\end{proof}

\begin{lemma}\label{product}
Assume that $\FF$ is a non-principal filter. Then $\FF\times 2^\omega\approx\FF$. 
\end{lemma}
\begin{proof}
Fix a $\Omega\in\FF\setminus\Cof(\omega)$. This is possible because $\FF$ is non-principal. Let $h:2^\Omega\times 2^{\omega\setminus\Omega}\longrightarrow 2^\omega$ be the function defined by setting $h(F,X)=F\cup X$. It is clear that $h$ is a homeomorphism. Furthermore, using the fact that $\Omega\in\FF$, one sees that  $h[\FF\re\Omega\times 2^{\omega\setminus\Omega}]=\FF$. Therefore $\FF\re\Omega\times 2^{\omega\setminus\Omega}\approx\FF$. An application of Lemma \ref{restriction} concludes the proof.
\end{proof}

\begin{lemma}\label{principal}
Assume that $\FF$ is a principal filter. Then $\FF^2\approx\FF$. 
\end{lemma}
\begin{proof}
It will be enough to show that $\FF\approx\QQQ$ or $\FF\approx\QQQ\times 2^\omega$. Fix $\Omega\subseteq\omega$ such that $\FF=\{X\subseteq\omega: \Omega\subseteq^\ast X\}$. If $\Omega\in\Cof(\omega)$, then $\FF=\Cof(\omega)\approx\QQQ$. So assume that $\Omega\notin\Cof(\omega)$. The proof of Lemma \ref{product} shows that $\FF\approx\FF\re\Omega\times 2^{\omega\setminus\Omega}$. Since $\FF\re\Omega=\Cof(\Omega)\approx\QQQ$, it follows that $\FF\approx\QQQ\times 2^\omega$.
\end{proof}

\section{The main result}

We begin by introducing some useful notation. Given $S\subseteq\omega$ such that $\omega\setminus S$ is infinite, let $\phi_S:\omega\setminus S\longrightarrow \omega$ denote the unique bijection such that $m<n$ implies $\phi_S(m)<\phi_S(n)$ for all $m,n\in\omega\setminus S$. Given an infinite $\Omega\subseteq\omega$, define
$$
\DD(\Omega)=\{(X,Y)\in 2^\Omega\times 2^\Omega:X\cap Y=\varnothing\}.
$$
It is easy to check that $\DD(\Omega)$ is a closed crowded subspace of $2^\Omega\times 2^\Omega$, which implies $\DD(\Omega)\approx 2^\omega$.

\begin{theorem}\label{main}
If $\FF$ is a filter then $\FF^2\approx\FF$.
\end{theorem}
\begin{proof}
Let $\FF$ be a filter. If $\FF$ is principal, then the desired conclusion follows from Lemma \ref{principal}. So assume that $\FF$ is non-principal, and fix $\Omega\in\FF\setminus\Cof(\omega)$.

Let $h:2^\Omega\times 2^\Omega\times \DD(\omega\setminus\Omega) \longrightarrow 2^\Omega\times \DD(\omega)$ be the function defined by
$$
h(F,G,X,Y)=(F\cap G, \phi_{F\cap G}[(F\setminus G)\cup X], \phi_{F\cap G}[(G\setminus F)\cup Y]),
$$
and observe that $h$ is continuous.

Let $g:2^\Omega\times  \DD(\omega)\longrightarrow  2^\Omega\times 2^\Omega\times \DD(\omega\setminus\Omega)$ be the function defined by
$$
g(H,Z,W)=(H\cup(\phi_{H}^{-1}[Z]\cap\Omega),H\cup(\phi_{H}^{-1}[W]\cap\Omega),\phi_{H}^{-1}[Z]\cap (\omega\setminus\Omega),\phi_{H}^{-1}[W]\cap (\omega\setminus \Omega)),
$$
and observe that $g$ is continuous. It is straightforward to verify that $g$ is the inverse function of $h$. Therefore $h$ is a homeomorphism.

Furthermore, it is easy to realize that
$$
h[\FF\re\Omega\times \FF\re\Omega\times \DD(\omega\setminus\Omega)]\subseteq\FF\re\Omega\times \DD(\omega)
$$
and
$$
g[\FF\re\Omega\times \DD(\omega)]\subseteq\FF\re\Omega\times \FF\re\Omega\times \DD(\omega\setminus\Omega).
$$
Since $g=h^{-1}$, it follows that $h[\FF\re\Omega\times \FF\re\Omega\times \DD(\omega\setminus\Omega)]=\FF\re\Omega\times \DD(\omega)$. Therefore $\FF\re\Omega\times \FF\re\Omega\times \DD(\omega\setminus\Omega)\approx\FF\re\Omega\times \DD(\omega)$. Finally, using Lemma \ref{restriction} and Lemma \ref{product}, one sees that $\FF^2\approx\FF$.
\end{proof}
\begin{corollary}
Fix natural numbers $m,n\geq 1$. If $\FF$ is a filter then $\FF^m\approx\FF^n$.
\end{corollary}

\section{Counterexamples for semifilters}

A \emph{semifilter} on $\Omega$ is a collection $\FF$ of subsets of $\Omega$ that satisfies conditions $(\ref{empty})$, $(\ref{finitemod})$, and $(\ref{superset})$. All semifilters are assumed to be on $\omega$. The following proposition shows that Theorem \ref{main} would not hold if condition $(\ref{intersection})$ were dropped from the definition of filter.

\begin{proposition}\label{semifilter}
There exists a semifilter $\TT$ such that $\TT^2\not\approx\TT$.
\end{proposition}
\begin{proof}
Fix infinite sets $\Omega_1$ and $\Omega_2$ such that $\Omega_1\cup\Omega_2=\omega$ and $\Omega_1\cap\Omega_2=\varnothing$. Define
$$
\TT=\{X_1\cup X_2:X_1\subseteq\Omega_1\text{, }X_2\subseteq\Omega_2\text{, Sand }(X_1\notin\Fin(\Omega_1)\text{ or }X_2\in\Cof(\Omega_2))\},
$$
and observe that $\TT$ is a semifilter. Furthermore, it is clear that $\TT$ is the union of its topologically complete subspace $\{X\subseteq\omega:X\cap\Omega_1\notin\Fin(\Omega_1)\}$ and its countable subspace $\{X_1\cup X_2:X_1\in\Fin(\Omega_1)\text{ and }X_2\in\Cof(\Omega_2)\}$.

The following two statements are easy to verify.
\begin{itemize}
\item $\Cof(\Omega_2)$ is a closed subspace of $\TT$ that is homeomorphic to $\QQQ$.
\item $\{X\subseteq\omega:\Omega_1\subseteq X\}$ is a closed subspace of $\TT$ that is homeomorphic to $2^\omega$.
\end{itemize}
It follows that $\TT^2$ has a closed subspace homeomorphic to $\QQQ\times 2^\omega$. Since, as is not hard to check, the space $\QQQ\times 2^\omega$ cannot be written as the union of a topologically complete subspace and a countable subspace, this concludes the proof.
\end{proof}

We remark that the semifilter $\TT$ in the above proof is actually homeomorphic to the notable space $\mathbf{T}$ introduced by van Douwen (unpublished, see \cite{vanengelenvanmill}). See \cite[Proposition 5.4]{medini} for more details.

In fact, the main result of \cite{medini} shows that every homogeneous  zero-dimensional Borel space that is not locally compact is homeomorphic to a semifilter. Together with \cite[Proposition 4.1]{vanengeleng}, which states that $\XX^2\not\approx\XX$ for almost every homogeneous zero-dimensional Borel space $\XX$ of low complexity, this yields many more counterexamples as in Proposition \ref{semifilter}.


\begin{thebibliography}{99}

\bibitem[vE1]{vanengelent}\textsc{A. J. M. van Engelen.} \emph{Homogeneous zero-dimensional absolute Borel sets.} CWI Tract, 27. Stichting Mathematisch Centrum, Centrum voor Wiskunde en Informatica, Amsterdam, 1986. iv+133 pp. Available at \verb"http://repository.cwi.nl/".

\bibitem[vE2]{vanengeleng}\textsc{F. van Engelen.} On Borel groups.
\emph{Topology Appl.} \textbf{35:2-3} (1990), 197--107.

\bibitem[vE3]{vanengeleni}\textsc{F. van Engelen.} On Borel ideals.
\emph{Ann. Pure Appl. Logic} \textbf{70:2} (1994), 177--203.

\bibitem[vEvM]{vanengelenvanmill}\textsc{F. van Engelen, J. van Mill.} Borel sets in compact spaces: some Hurewicz type theorems. \emph{Fund. Math.} \textbf{124:3} (1984), 271--286.

\bibitem[Me]{medini}\textsc{A. Medini.} On Borel semifilters. Preprint. Available at \verb"http://arxiv.org/abs/1605.01024".

\bibitem[vM]{vanmill}\textsc{J. van Mill.} \emph{The infinite-dimensional topology
of function spaces.} North-Holland Mathematical Library, 64. North-Holland
Publishing Co., Amsterdam, 2001. xii+630 pp.

\bibitem[St]{steel}\textsc{J. R. Steel.} Analytic sets and Borel isomorphisms.
\emph{Fund. Math.} \textbf{108:2} (1980), 83--88.

\end{thebibliography}
\end{document}